\documentclass[11pt,regno]{amsart}

\usepackage{overpic}
\usepackage{dsfont}
\usepackage{euscript,graphicx,color}
\usepackage{amsmath}
\usepackage{amssymb}
\usepackage{amsfonts}
\usepackage{mathrsfs}
\usepackage{epstopdf,rotating}

\newtheorem{theo}{Theorem}[section]
\newtheorem{theo-app}{Theorem}[section]

\newtheorem{theorem}[theo]{Theorem} 
\newtheorem{corollary}[theo]{Corollary}
\newtheorem{proposition}[theo]{Proposition}

\newtheorem{lemma}[theo]{Lemma}
\theoremstyle{definition}
\newtheorem{example}[theo]{Example}
\newtheorem{remark}[theo]{Remark}
\newtheorem*{remark*}{Remark}
\newtheorem{definition}[theo]{Definition}

\numberwithin{equation}{section}

\renewcommand{\hat}{\widehat}

\newcommand{\sO}{\mathscr{O}}

\newcommand{\sT}{\mathscr{T}}
\newcommand{\sA}{\mathscr{A}}
\newcommand{\sB}{\mathscr{B}}

\newcommand{\I}{\mathcal{I}}

\DeclareMathOperator{\cl}{cl}

\DeclareMathOperator{\BD}{BD}

\newcommand{\R}{\mathbb{R}}

\newcommand{\C}{\mathbb{C}}

\DeclareMathOperator{\diam}{diam}

\DeclareMathOperator{\dist}{Dist}
\DeclareMathOperator{\interior}{int}

\DeclareMathOperator{\Crit}{{\rm Crit}}

\DeclareMathOperator{\Comp}{{\rm Comp}}

\DeclareMathOperator{\spanning}{{\rm spanning}}
\DeclareMathOperator{\tree}{{\rm tree}}
\DeclareMathOperator{\Const}{{\rm Const}}

\DeclareMathOperator{\ddist}{{\rm dist}}

\def\bR{\mathbb{R}}

\def\e{\varepsilon}

\DeclareMathSymbol{\varnothing}{\mathord}{AMSb}{"3F}
\renewcommand{\emptyset}{\varnothing}

\author[F. Przytycki]{Feliks Przytycki$^\dag$} \address{Institute of Mathematics, Polish Academy of Sciences, ul. \'{S}niadeckich 8, 00-956 Warszawa, Poland}
\email{feliksp@impan.pl}

\thanks{
$\dag$Partially supported by Polish NCN Grant "Chaos, fractals and conformal dynamics, III",
2014/13/B/ST1/01033}

\begin{document}

\date{\today}

\title[geometric pressure]{Geometric pressure in real and complex 1-dimensional dynamics via trees of pre-images and via spanning sets}

\begin{abstract}
We consider $f:\hat I\to \R$ being a $C^3$ (or $C^2$ with bounded distortion) real-valued multimodal map with non-flat critical points, defined on $\hat I$ being the union of closed intervals, and its restriction to the maximal forward invariant subset $K\subset I$. We assume that $f|_K$ is topologically
transitive. We call this setting the (generalized multimodal) real case. We consider also $f:\C\to \C$
a rational map on the Riemann sphere and its restriction to $K=J(f)$ being Julia set (the complex case).
We consider topological pressure $P_{\spanning}(t)$ for the potential function $\varphi_t=-t\log |f'|$ for $t>0$ and iteration of $f$ defined in a standard way using $(n,\e)$-spanning sets. Despite of $\phi_t=\infty$ at critical points of $f$, this definition makes sense (unlike the standard definition using $(n,\e)$-separated sets) and we prove that $P_{\spanning}(t)$ is equal to other pressure quantities, called for this potential
{\it geometric pressure}, in the real case under mild additional assumptions,  and in the complex case provided there is at most one critical point with
forward trajectory accumulating in $J(f)$.  $P_{\spanning}(t)$ is proved to be finite for  general rational maps, but it may occur infinite in the real case. We also prove that geometric tree pressure in the real case is the same for trees rooted at all `safe' points, in particular at all points except the set of Hausdorff dimension 0, the fact missing in \cite{PrzRiv:13}, proved in the complex case in \cite{Prz:99}.

\end{abstract}

\maketitle
\tableofcontents

\section{Introduction}

Let us start with the classical
\begin{definition}[Topological pressure via separated sets]

Let $f:X\to X$ be a continuous map of a compact metric space $(X,\rho)$ and $\phi:X\to \R$ be a real continuous function. For every positive integer  $n$ and $x\in X$ denote $S_n\phi(x)=\sum_{j=0}^{n-1}\phi(f^j(x))$.
For every integer $n\ge 0$ define the metric $\rho_n(x,y)=\max\{\rho(f^j(x),f^j(y)): j=0,...,n\}$.
For every $\e>0$ a set $Y\subset X$ is said to be $(n,\e)$-separated if for every $y_1,y_2\in Y$ such that $y_1\not=y_2$ it holds $\rho_n(y_1,y_2)\ge \e$. Define

\begin{equation}
P_{\rm {sep}}(f,\phi,\e):=\limsup_{n\to\infty}\frac1n \log \big(\sup_Y \sum_{y\in Y}\exp S_n\phi(y)\big),
\end{equation}

supremum taken over all $(n,\e)$-separated sets $Y\subset X$, and

$$
P_{\rm {sep}}(f,\phi)= \lim_{\e\to 0} P_{\rm {sep}}(f,\phi,\e).
$$
\end{definition}

Analogously
\begin{equation}
P_{\rm {spanning}}(f,\phi,\e):=\limsup_{n\to\infty}\frac1n \log \big(\inf_Y \sum_{y\in Y}\exp S_n\phi(y))\big),
\end{equation}
infimum taken over all $(n,\e)$-spanning sets $Y\subset X$, i.e. such that for every
$x\in X$ there exists  $y\in Y$ such that $\rho_n(x,y) < \e$, in other words such that
$\bigcup_{y\in Y}B_{\rho_n}(y,\e)=X$,
and
$$
P_{\rm {spanning}}(f,\phi)= \lim_{\e\to 0} P_{\rm {spanning}}(f,\phi,\e).
$$

It is easy to prove
\begin{theorem}[see e.g. \cite{Walters}]
For every continuous $\phi:X\to \R$
\begin{equation}
P_{\rm {spanning}}(f,\phi)=  P_{\rm {sep}}(f,\phi).
\end{equation}
This pressure depends on topology, but does not depend on metric.
\end{theorem}
This equality follows from
\begin{equation}\label{equal-pressures}
P_{\rm {sep}}(f,\phi,2\e) \le P_{\rm {spanning}}(f,\phi,\e)\le P_{\rm {sep}}(f,\phi,\e).
\end{equation}

 \

 In this paper we shall discuss  $\phi=\phi_t= -t\log |f'|$
for all parameters $t>0$. This is slightly different from the previous situation in dimension 1 if $f$ is differentiable and has critical points, i.e. points $c\in X$ where the derivative of $f$ is 0. At these points $\phi$  is not continuous. We assign to $\phi$ the value $+\infty$ there. The notion $P_{\rm sep}$
does not make much sense in this case, as this quantity is equal to $+\infty$, by taking $Y$ containing some critical points, so it is replaced by the notion of tree pressure, see \eqref{Ptree}. However $P_{\rm spanning}$ defined as above happens to make sense and a part of this paper is devoted to explaining this.

\

We discuss two settings:

1. (Complex) $f$ is a rational mapping of degree at least 2 of the Riemann sphere $\widehat{\C}$ usually with the spherical metric. We consider $f$ on its Julia set $K=J(f)$.

2. (Real) $f$ is a real generalized multimodal map. Namely it is defined on a neighbourhood ${\bf U}\subset\bR$ of its invariant set $K$. We assume $f\in C^2$,
is non-flat at all its turning and inflection critical points, has bounded distortion property for its iterates, and is topologically transitive and has positive
topological entropy on $K$.

We assume that $K$ is a maximal invariant subset on a finite union of pairwise disjoint closed intervals $\hat I=\bigcup_j\hat I_j\subset {\bf U}$ whose ends are in $K$. (This maximality corresponds to Darboux property.)
By adjusting $\hat I$ and ${\bf U}$ we can assume there are no critical points outside $K$, no attracting  periodic orbits in ${\bf U}$ and no parabolic periodic orbits in ${\bf U}\setminus K$.

We write $(f,K)\in {\sA}^{\BD}_+$. The subscript + is to mark positive topological entropy. Sometimes we write $(f,K,\hat I,{\bf U})$. (In place of BD one can assume $C^3$ and all periodic orbits in $K$ hyperbolic repelling, denote the related class by ${\sA}^{3}_+$.)

For this real setting see \cite{PrzRiv:13}. Examples: sets in the spectral decomposition \cite{dMvS}.

\smallskip

In both settings the set of all critical points will be denoted by $\Crit(f)$.

\

The function $\phi_t$ is sometimes called {\it geometric potential} and the pressure is called {\it geometric pressure}, see e.g. \cite{PrzUrb:10}.

\smallskip

This name is justified by
$$
\exp S_n\phi_{t_0} (z)=|(f^n)'(z)|^{-t_0}\approx (\diam B_n(z))^{t_0},
$$
where $B_n(z):=\Comp_z(f^{-n}(B(f^n(z), \Delta)))$ for a constant $\Delta>0$ and appropriate $t_0$ (with pressure $P(\phi_{t_0})=0$). Here  $\Comp_z$ means the component in $\C$ or $\R$ containing $z$, called also a {\it pull-back}. We consider only pull-backs intersecting $K$.

\smallskip

There are several equivalent definitions of geometric pressure $P(\phi_{t})$, see \cite{PrzRivSmi:04}, \cite{PrzUrb:10} or \cite{PrzRiv:13} in the interval case.
One of them useful in this paper is
\begin{definition}[ hyperbolic pressure]\label{hyperbolic-pressure}
$$
P_{\rm{hyp}}(t):= \sup_X P(f|_X,-\phi_t|_X) ,
$$
supremum taken over all compact $f$-invariant (that is $f(X)\subset X$) isolated
(Cantor) uniformly hyperbolic subsets of $K$.

{\it Isolated} (or \emph{forward locally maximal}), means that there is a neighbourhood $U$ of $X$ such that
$f^n(x)\in U$ for all $n\ge 0$ implies $x\in X$.

A set $X$ is said to be {\it hyperbolic} or {\it expanding} if there is a constant $\lambda_X>1$
such that for all $n$ large enough and all $x\in X$ we have
$|(f^n)'(x)|\ge\lambda_X^n$.
\end{definition}

\

\section{Tree-pressure}

\subsection{Definitions}

We devote this section to studying a modified definition of pressure by separated sets which may not have sense,
called tree-pressure, see e.g. \cite{PrzUrb:10}.

In the real and complex settings we define

\begin{multline}\label{Ptree}
 P_{\rm tree}(z,t):=\limsup_{n \to \infty} \frac1n\log Q_n(z,t), \\
 \hbox{where}\ \ Q_n(z,t):=\sum_{y \in f^{-n}(z)\cap K} |(f^n)'(y)|^{-t}.
\end{multline}

\begin{definition}\label{safe}
A point $z\in K$ is said to be
\textit{safe}, or $\Crit(f)$-{\it safe} if
for every $\delta>0$ and all $n\ge n(\delta> 0$ large enough
\begin{equation}\label{safe1}
B(z, \exp (-\delta n))\cap \bigcup_{j=1}^n(f^j(\Crit(f))=\emptyset.
\end{equation}
\end{definition}

\smallskip

It immediately follows from the definitions that Hausdorff dimension of the set of points which are not safe is equal to 0.

\smallskip

In the complex setting the following holds

\begin{theorem} [\cite{Prz:99}, \cite{PrzRivSmi:04}]\label{tree-hyp-complex}
For every rational $f:\widehat{\C}\to\widehat{\C}$ of degree at least 2 and for every $z\in K=J(f)$ safe and $t>0$ it holds
$$
P_{\rm tree}(z,t)=P_{\rm hyp}(f,\phi_t).
$$
\end{theorem}

In particular in the complex case $P_{\rm tree}(z,t)$ does not depend on $z$ safe; it is constant except $z$ in a set of Hausdorff dimension 0. We denote this tree-pressure for $z$ safe by $P_{\tree}(t)$. 

\subsection{The real case: independence of a safe point}

In the generalized multimodal setting the above equality was known for $z$ being safe, safe forward (in case $K$ is not weakly isolated) and hyperbolic, see
\cite{PrzRiv:13}[Lemma 4.4]. We remind the definitions mentioned here, compare \cite{PrzRiv:13}:

\begin{definition} [hyperbolic]\label{hyperbolic} A point $z\in K$ is called hyperbolic  (or expanding)if there exist $\lambda>1$ and $\Delta>0$ such that for all $n>0$ \,  
$|(f^n)'(z)|\ge \Const\lambda^n$ and  $f^n$ maps diffeomorphically 
$\Comp_z(f^{-n}(B(f^n(z), \Delta)))$
onto $B(f^n(z), \Delta)$. 
\end{definition}

\begin{definition}[safe forward]\label{safe forward}
A point $z\in K$ is called \textit{safe forward} if there exists $\Delta>0$ such that $\ddist(f^j(z), \partial \hat I_K)\ge \Delta$ for all $j=0,1,...$.
\end{definition}

\begin{definition}[weak isolation]\label{weak isolation}
A compact set $K \subset\R$ is said to be {\it weakly isolated} for a continuous mapping on a neighbourhood of $K$ to $\R$ for which $K$ is forward invariant,
if there exists $\e>0$
such that every  $f$-periodic orbit $O(p)\subset B(K,\e)$ must be in $K$.
\end{definition}

Though the set of all expanding points has Hausdorff dimension equal to the hyperbolic dimension of $K$, i.e. supremum of Hausdorff dimensions of  isolated
uniformly hyperbolic subsets of $K$, being the first zero of the hyperbolic pressure, see the definition above,
the complementary set can also be large.

One of aims of this paper is to prove

\begin{theorem}\label{Independence-safe} For $(f,K)\in {\sA}^{\BD}_+$ without parabolic periodic orbits
(or for $f\in \sA^3_+$),
if $K$ is weakly isolated and $t>0$,
the tree pressure $P_{\rm tree}(z,t)$ does not depend on $z\in K$ safe. In particular
\begin{equation}\label{tree=hyp}
P_{\rm tree}(z,t)=P_{\rm hyp}(f,\phi_t).
\end{equation}
Moreover limsup can be replaced by lim in the definition of tree pressure, i.e the limit exists.
\end{theorem}

As in the complex case we denote this tree-pressure for safe points by $P_{\tree}(t)$.

\smallskip

Before proving this theorem let us recall the following definition valid in the real and complex cases

\begin{definition}
[backward Lyapunov stable]\label{backward Lyapunov}
$f$ is said to be {\it backward Lyapunov stable} if for every $\e>0$ there
exists $\delta>0$ such that for every $z\in K, n\ge 0$ and $W=\Comp_z f^{-n}(B(f^n(z),\delta)$ (the balls and components in $\R$ or $\C$)
$\diam W <\e$.
\end{definition}

In the real case this property holds for $B(f^n(z), \delta)$ not containing any parabolic periodic point, see \cite{PrzRiv:13}[Lemma 2.10].

\smallskip

In the sequel we call $W$ a pull-back of the interval $B(f^n(z),\delta)$ for $f^n$ containing $z$.

\begin{proof}[Proof of Theorem \ref{Independence-safe}]

The inequality $P_{\rm tree}(z,t) \ge P_{\rm {hyp}}(f,\phi_t)$ is obviously true for every $z\in K$, under a mild non-exceptionality condition, weaker than safe, see \cite{PrzRiv:13}[Lemma 4.4].

\smallskip

Thus,  it is enough to prove that for all $w, z\in K$ both being safe it holds
\begin{equation}\label{ineqtree}
P_{\rm tree}(w,t)\le P_{\rm tree}(z,t).
\end{equation}

It follows from the topological transitivity that given any $\delta'>0$ there exists
$N(\delta')$ such that
$A=A(z,\delta'):=\bigcup_{j=0,...,N}f^{-j}(z)\cap K$ is $\delta'$-dense in $K$ (i.e. $\bigcup_{y\in A}B(y,\delta')\supset K$; in other words $(0,\delta')$-spanning), see \cite{PrzRiv:13}[Remark 2.6, Proposition 2.4].

\smallskip

Fix an arbitrary $0<\delta'\le \delta$, where $\delta$ is small, chosen to $\e$ in Definition \ref{backward Lyapunov}, and $\e$ satisfies the weak isolation condition \ref{weak isolation}.

\smallskip

We have two cases:

Case 1.  There exist $z_0,z'_0\in A(z,\delta')$ such that $z_0\le w \le z'_0$ and $|z_0-z'_0|<\delta$.

\smallskip

In this case, if $\delta$ is small enough,  as in
Definition of backward Lyapunov stability, all pull-backs of $W_0=[z_0,z'_0]$ are
shorter than $\e$.
Now we use a procedure by Rivera-Letelier \cite{R-L:12}, see also \cite{GPR:14}[Appendix C]. We consider
pull-backs of $\hat W_0=2W_0$ (the twice longer interval with the same origin) for $f^i$ along a backward trajectory  until for certain $i_1$ the pull-back $\hat W^1_{i_1}$ for $f^{i_1}$ captures a critical point.
Next we consider pull-backs of $2W_{i_1}$, where $W_{i_1}\subset \hat W^1_{i_1}$ is a pull-back of $W_0$.   We arrive after time $i_2-i_1$ at
$\hat W^2_{i_2}$  containing a critical point, etc.

Using bounded distortion between consecutive captures of critical points, more precisely for
$f^{i_2-i_1-1}$ on $f(W^2_{i_2})$,  and the inequality for every $x\in W_i$ true for any pull-back of $W_0$ for $f^i$

\begin{equation}\label{one-step}
|f'(x)|\le \Const \frac{\diam W_i}{\diam W_{i-1}}
\end{equation}
we prove that 
for $z_n\in\partial W_n$
and $\alpha>0$ arbitrarily close to 0 for $\delta$ small enough that the differences of times of consecutive captures of each critical points are bounded below by a constant arbitrarily large (possible due to absence of attracting periodic orbits),

$$
|(f^n)'(z_n)|^{-t}\ge (\exp -\alpha n) \Bigl(\frac{\diam W_n}{\diam W}\Bigr)^t
$$

Since $w$ is safe we get also, replacing $W_n$ by the appropriate pull-back $V_n$ of
$V=B(w,\exp (-\e_1 n))\cap W_0$, since $f^n$ is invertible of bounded distortion on $V_n$,
$$
({\diam V_n})^t \ge
\Const \exp (-t\e_1 n) |(f^n)'(w_n)|^{-t}.
$$
Hence
$$
|(f^n)'(z_n)|^{-t}\ge (\exp -(\alpha+t\e_1) n)|(f^n)'(w_n)|^{-t}.
$$

Summing over all $n$'th preimages $w_n$ of $w$ in $K$, taking in account that the number
 of $w_n$'s can be at most $\exp \e_2 n$ in each pull-back $W_n$, gives the demanded estimate.
 More precisely:
 \[\begin{split}
 Q_n(w,t)\le  & 2Q_n(z_0,t) + 2Q_n(z_0',t) \\
 & \le 4 \max_{j=0,...,N} Q_{n+j}(z,t) L^{jt}\le 
 4Q_{n+N}(z,t)L^{Nt}
 \end{split}\]
 for $L=\sup |f'|$. Acting with $\frac 1n \log$ and passing with $n$ to $\infty$ in limsup  we get the inequality
 (\ref{ineqtree})  for $w$ and $z$ hence after the interchanging their roles, the equality.
 Similarly we obtain the equality of lower limits, writing 
 $Q_n(z,t)\ge \frac14 L^{-Nt} Q_{n-N}(w,t)$. 
 
 But they coincide with $P_{\rm {hyp}}(t)$ for $z$ safe, safe forward and hyperbolic, see the beginning of this section. Hence limsup and liminf coincide and are equal to $P_{\rm {hyp}}(t)$ for every $z$ safe.

\

It is a priori not clear whether the points $z_n$ belong to $K$. This trouble can be dealt with
as follows,
compare the proof of \cite{GPR:14}[Lemma 3.2]. 
Take an arbitrary repelling periodic not post-critical orbit ${\sO}\subset K$ and a backward trajectory of a point $p \in \sO$ in $K$ accumulating at $z_0$ (and $z'_0$). To simplify notation we can assume that $p$ is a fixed point for $f$. 

Consider  $w_n\in [z_n,z'_n]=W_n$ as above.  
Choose $w_{n+N_p}\in f^{-N_p}(w_n)\cap B(p,r_p)$, where $r_p$ is such that there exists a branch $g$ of $f^{-1}$ with  $g(p)$,  mapping $B(p,r_p)$ into itself, with its iterates converging to $p$. Since $z$ 
is safe, $z_0$ (and $z_0'$) is not postcritical. 

Next choose a backward trajectory $(y_0,y_1,...)$ of $p$ so that $z_0,z_0'$ are its limit points.
Choose the intervals $B=B(z_0,\xi)$ and $B'=B(z_0',\xi)$ so short that the pull-back $B_{n+N_p}$ of the one of them containing $z_{n+N_p}$ is in $B(p,r_p)$ and $f^{n+N_p}$ has no turning critical points in it. Next choose $r'$ and $n'$ such that a pull-back $W''$ of $B(p,r')$ for $f^{n'}$ is in $B$ (or $B'$). Finally choose $m$ such that $g^m(B(p,r_p))\subset B(p,r')$. So the adequate branch $G$ of $f^{-(n+N_p+m+n')}$ maps $B$ (or $B'$) into itself, so a corresponding fixed point $p_\xi$ for $f^{n+N_p+m+n'}$ exists in $B$ (or $B'$). Write $f^n(z_n)=z_0$. A part of the periodic trajectory of $p_\xi$ shadows the backward trajectory $(z_0,...,z_{n+N_r})$. By the weak isolation property $p_\xi\in K$. The shadowing error tends to 0 as $\xi\to 0$. Thus $z_n\in K$.



\

Case 2. The safe point $w\in K$  is not between two points $z_0,z_0'$, in notation of Case 1. We assume $\delta'\le \delta/4$. Then
 the interval $(w-\delta+2\delta', w-\delta')$ (or $(w+\delta',w +\delta-2\delta')$) is disjoint from $K$.
 Call a component of $\R\setminus K$ of length at least $\delta/4$ a {\it large gap}. By the boundness of $K$ there are at most a finite number $\Gamma_\delta$ of them. Denote the union of the boundary points of large gaps by $\partial G$ and the union of the large gaps by $G$. For $\delta'$ satisfying  additionally
 \begin{equation}\label{delta'}
 L\delta'+\delta'<\min_{x\in\partial G} \ddist (f(\partial G)\setminus \partial G, \partial G).
 \end{equation}
 implying
 $\bigcup_{x\in \partial G, f(x)\notin\partial G} B(f(x),\delta'))\cap B(\partial G, \delta')=\emptyset$,  we conclude that either for some $m:0<m\le 2\Gamma_\delta$, $f^m(w)$ satisfies the assumption of Case 1 (is between $z_0,z_0'$), or all $f^j(w)$ for $0\le j<2\Gamma_\delta$ are $\delta'$ close to
 $\partial G$ hence $w$ is pre-periodic $w':=f^{j_1}(w)=f^{j_2}(w)$, and the length of its forward orbit is bounded by $2\Gamma_\delta$.

Then use $\hat z\in f^{-\kappa n}(z)$  which is $\exp -\eta n$ close to $w$ i.e. in a "safe" ball, for $0<\eta < \kappa\chi$, where $\chi$ is Lyapunov exponent at $w$. Taking $\kappa$ arbitrarily small (positive) we can replace $z_0$ by $\hat z$ when comparing the tree pressure at $z$ and $w$. We use  $|f'|\le L$ and $t\ge 0$. We use also the fact that by the safety condition the distortion
$|(f^n)'(\hat z_n)|/|(f^n)'(w_n)|$ is uniformly bounded for $w_n$. 

(This allows not to use $\hat z'\in f^{-k}(z)$ on the other side of $w$ maybe not existing for $k$ of order at most $\kappa n$.)

\end{proof}

 \subsection{On the weak isolation condition}

 Notice that proving that $z_n\in K$ above, we use the existence of $z=z_{n+N_r}$ such that $f^{N_r}(z)=z_n$. A priori we cannot exclude that only $z_n'$ happens in the images. An example is $f(x)=ax(1-x)$ for $a<4$ close to 4, on a neighbourhood $U$ of $\hat I=[f^2(1/2),f(1/2)]$. Then points slightly to the left of $f^2(1/2)$ may have no preimages.

 Fortunately in the proof above we use only those $z_i$ which are boundary points of pull-backs of $[z_0,z_0']$ containing $w_i\in K$.

 \smallskip

 In the above proof to know that $z_n,z_n'$ belong to $K$ we could refer to \cite{GPR:14}[Corollary 3.3] in the form of the proposition below (interesting in itself), true under the additional assumption, see \cite{GPR:14}[Subsection 1.4], that no point in $\partial \hat I$ is weakly $\Sigma$-exceptional, for $\Sigma$ being the set of all turning critical points.

\begin{definition}
Given an arbitrary finite set $\Sigma\subset K$, we call a nonempty set $E\subset K$ \emph{weakly $\Sigma$-exceptional}, if $E$ is non-dense in $K$ and satisfies
\begin{equation}\label{except-condition}
	(f|_K)^{-1}(E)\setminus  \Sigma \subset E
	.
\end{equation}
We call $x\in K$ weakly $\Sigma$-exceptional
 if it is contained in a weakly $\Sigma$-exceptional set.
\end{definition}


\begin{proposition}[On K-homeomorphisms]\label{prop:K}
Let $(f,K)\in \sA_+$ satisfies weak isolation condition.
  Let $W$ be an arbitrary interval sufficiently short (closed, half-closed or open),
  not containing in its closure weakly $\Sigma$-exceptional points for $\Sigma$ being the set of turning critical points in $\partial\hat I$,
  such that $f$ is monotone on $W$ and $\cl W\cap K\not=\emptyset$. Then $f|_W$ is a $K$-homeomorphism, that is $f(W\cap K)=f(W)\cap K$.
\end{proposition}

\begin{proof} It is sufficient to consider $W$ closed. The assertion of the Proposition follows for
$W':= W\cap \hat I=W\cap\hat I_j$ by the maximality of $K$
(notice that $W$ short enough intersects only one interval $\hat I_j$).
By definition $W'':=W\setminus W'$ is disjoint from $K$. For $W$ short enough $W''$ has one component or it is empty. Suppose it is non-empty. The case $f(W'')$
intersects $K$, but $f(a)$ is not a limit point of $f(W'')\cap K$ can be eliminated by considering $W$ short enough. Here $a$ is the boundary point of $\hat I_j$ belonging to $\cl W''$. We use the assumption that
the family $\hat I_j$ is finite.

Therefore we need only to consider the case $f(a)$ is an accumulation point of $f(W'')\cap K$ (in particular $f|_K$ is not open at $a$). In this case however there exists a periodic orbit $Q$ passing through $W''$ arbitrarily close to $K$. The proof is
the same as the proof of \cite{GPR:14}[Lemma 3.2] and similar to the proof of Theorem \ref{Independence-safe}.
Briefly: we choose a repelling periodic orbit
$\sO \subset K$.
Next choose a backward trajectory $(y_0,y_1,...)$ of a point $p\in \sO$ with a limit point in $f(W'')\cap K$ and a backward trajectory $(z_0,z_1,...)$ of $a$ converging to $Q$. This allows us to find a backward trajectory of $W''$  at a time $n$ approaching to $\sO$ along $z_j$ and next at a time $m$ being in $f(W'')$. So $W''$ after the time $n+m+1$ enters itself.
Hence there exists a branch of $f^{-(m+n+1)}$ mapping $W''$ into itself, yielding the existence of $Q$.

So $Q$ is in $K$ by the weak isolation condition.
We obtain a point in $K\cap W''$, a contradiction.
\end{proof}

Remark that in the example $f(x)=ax(1-x)$ discussed above the assumption of the lack of weakly $\Sigma$-exceptional points does not hold and the assertion of Proposition \ref{prop:K} fails for $W=[f^2(1/2)-\delta,f^2(1/2)]$.

\

\section{Geometric pressure via spanning sets. The complex case}

In the real case in the previous section we used  the property: backward Lyapunov stability,
Definition \ref{backward Lyapunov}. In the complex case this property need not hold.
 So
the following weaker version occurs useful.

\begin{definition}\label{weak Lyapunov}
A rational mapping $f:\widehat{\C}\to\widehat{\C}$
is said to be {\it weakly backward Lyapunov stable} {\bf wbls}, if
for every $\delta>0$ and $\e>0$ for all $n$ large enough and every disc $B=B(x,\exp -\delta n)$ centered at $x\in J(f)$, for every $0\le j \le n$ and every component $V$ of $f^{-j}(B)$ it holds $\diam V\le\e$.
\end{definition}

Denote $P_{\rm spanning}(f|_K, \phi_t)$ by $P_{\rm spanning}(t)$, both in the real and complex case.
The following is the main theorem in this section

\begin{theorem} \label{spanning-complex}
For every rational mapping $f:\widehat{\C}\to\widehat{\C}$ of degree at least 2 and for every  $t>0$
it holds
$P_{\spanning}(t) \ge P_{\tree}(t)$.
If $f$ is weakly backward Lyapunov stable
then the opposite inequality holds, hence
\begin{equation}
P_{\spanning}(t) = P_{\tree}(t) 
\end{equation}
\end{theorem}

\begin{proof}

{\bf I.} First we prove
$P_{\spanning}(t) \le P_{\tree}(t).$ This is the CONSTRUCTION part of the proof, where we construct an $(n,\e)$-spanning set not carrying much more "mass" than $f^{-n}(\{z_0\})$.
This corresponds to the right hand side inequality in (\ref{equal-pressures}), where we can just consider
as the $(n,\e)$-spanning sets  maximal $(n,\e)$-separated sets.

\smallskip

Fixed an arbitrary $\e>0$ and $\delta>0$, by the property {\bf wbls} (Definition \ref{weak Lyapunov}) we have  for $n$ large enough for every $x\in J(f)$ and every pull-back $V$ of $B(x,\exp(-n\delta/2))$ for $f^j, j=0,...,n$,\ $\diam V<\e$.

Denote $\sB:=\bigcup_{c\in\Crit(f)\cap J(f)}\bigcup_{j=1,...n} B(c,j)$,
where $B(c,j):=B(f^j(c), r)$ where $r:=\exp (-n\delta))$ as in the safety assumption.

We can easily find a set $X\subset J(f)\setminus \sB$ which is $r/2$-spanning for $\rho$ the standard metric on the Riemann sphere and
$\#X\le \Const \exp 2n \delta $.

The set $B(X, r/2)$ covers $J(f)\setminus \sB$. By bounded distortion for every $x\in X$ and $x'\in B(x,r/2)$ and branch $g$ of $f^{-n}$ on $B(x,r/2)$, $C_{\dist}^{-1}\le \frac{|g'(x)|}{|g'(x')|}\le C_{\dist}$.

\smallskip

Let $B^1,B^2,...B^N$ be all the components of $\sB$.

\smallskip

Assume first for simplicity that $J(f)$ is connected.

Clearly for every $1\le k\le N$,\ $\diam B^k\le r\cdot n\#(\Crit(f)\cap J(f))$.
By the connectivity of $J(f)$,\ there exists $x^k\in \partial B^k\cap J(f)$ if $n$ is large enough.
For $n$ large we have also $\diam B^k<\exp (-n\delta/2)$. Hence the diameters of all pull-backs of $B^k$ for $f^j, j=1,...,n$ are less than $\e$. Let $\hat X=X\cup \bigcup_k\{x^k\}$. Then $Y=f^{-n}(\hat X)$ is $(n,\e)$-spanning.

We have for an arbitrary $\xi>0$
\begin{equation}\label{key-spanning}
\sum_{y\in Y} |(f^n)'(y)|^{-t} \le \#(\hat X) \exp n(P_{\tree}(t)+\xi)
\end{equation}
for $n$ large enough.
This uses the fact that the convergence in \eqref{Ptree} is uniform for all $x$ safe with the same $\delta$, see Lemma \ref{unif}. (Here we consider $x\in \hat X$ depending on $n$, so we abuse the terminology; we consider safe for each $n$ separately, just satisfying (\ref{safe1}).) Considering $n\to \infty$ and passing with $\delta$ and $\e$ to 0, we end the proof.

\smallskip

Now consider the general case, allowing $J(f)$ being disconnected.

\phantom\qedhere\end{proof}

\begin{definition}
A compact set $X\subset\C$ in the complex plane is said to be {\it uniformly perfect} if there exists $M>0$ such that there is no annulus $D\subset \C$ of modulus bigger than $M$, separating $X$. In other words
there is no $A=\{z\in\C: r_1<|z-z_0|<r_2\}$ such that
$\log \frac{r_2}{r_1}>M$, and $X\cap\{|z-z_0|\le r_1\}\not=\emptyset$ and
$X\cap\{|z-z_0|\ge r_2\}\not=\emptyset$.
\end{definition}

\begin{lemma}\label{uniformlyperfect}
Let $X\subset\C$ be a compact uniformly perfect set. Then there exists $\kappa>0$ such that for every
$0<a\le 1$, every $m$ large enough and every $\hat x\in X$ there exists in the Euclidean metric an
$\exp( -m)$-separated set $X_{m,a}\subset B(\hat x,\exp -(1-a)m)$ such that
$\# X_{m,a}\ge \exp \kappa am$.
\end{lemma}
\begin{proof} We can assume $\diam X\ge 2$.
If $X$ is uniformly perfect with constant $M$, then for every $x\in X$ and $i\ge 0$ such that $(3i+1)M<am$ we can find $x'_i\in X$ such that
$$
\exp (-m+3iM) \le |x-x'_i|\le \exp (-m+(3i+1)M).
$$
Now we define $X_{m,a}$. Each its element will be encoded by a block of 0's and 1's of length $\I+1$ where
$\I:= [am/3M]-1$ where $[\cdot]$ means the integer part, and
denoted $x(\nu_0,...,\nu_{\I})$, where $\nu_j=0$ or $1$. We define these points by induction using codings
of length $1,2,...$  and finally $\I+1$ which will be our final choice.

For $x=\hat x$ set $x_0(0)=x$ and $x_0(1)=(x)'_{\I}$. The subscript at $x$ denotes the length of the block of the coding symbols minus 1, here length 1.
Having defined $x_i( \nu_0,...,\nu_i)$ define
\[\begin{split}
x_{i+1}( \nu_0,...,\nu_i,0) & =x_i( \nu_0,...,\nu_i) \, \, \, \hbox{and}\,\, \\ &
x_{i+1}( \nu_1,...,\nu_i,1)  =(x_i( \nu_0,...,\nu_i))'_{\I-(i+1)}.
\end{split}\]

Notice that all the points $x(\nu_1,...,\nu_{\I})$ are within the distance at least $\exp(-m)$ from each other.
Indeed, if $0\le i \le \I$ is the first index with digits $\nu_i$ different for two such points $y=x( \nu_0,...,\nu_{\I})$ and
 $z=x( \nu^{\ast}_0,...,\nu^{\ast}_{\I})$, then, say $\nu_i=0, \nu^{\ast}_i=1$,   by construction, setting $x_{-1}=x_0$ for $i=0$,
\begin{multline*}
|y-z| \ge |x_{i}( \nu_0,...,\nu_{i-1},0)|  \\ - (x_{i}( \nu_0,...,\nu_{i-1},1)| -
|y-x_{i}( \nu_0,...,\nu_{i-1},0)| -
|z-x_{i}( \nu_0,...,\nu_{i-1},1)|
\\
\ge
\exp (-m+3(\I-i)M) - 2\sum_{s=i+1}^{\I} \exp(-m+3(\I-s)M) \ge
\\
\exp (-m) \bigl(\exp 3(\I-i)M - 2 \frac{\exp (3(\I-i)M)-1}{\exp (3M)-1}\bigr)
\ge \exp(-m)
\end{multline*}
for $M\ge 1$.

Thus, define $X_{m,a}=\{x(\nu_0,...,\nu_{\I})\}$. Notice that
$$
\# X_{m,a}=2^{[am/3M]+1}=\exp( (\log 2) [am/3M]+1)\ge
\exp \kappa am
$$
with $\kappa$ arbitrarily close to $(\log 2)/3M  $ for $am$ large enough. The corresponding assertion of the lemma is proved.

Finally notice that for each $y\in X_{m,a}$
$$
|\hat x-y|< \exp (-m+(3\I+2)M)\le \exp (-m +am) = \exp -(1-a)m.
$$

\end{proof}

\begin{proof}[Continuation of Proof of Theorem \ref{spanning-complex}] We deal now with the non-connected
$J(f)$ case.
Let $m:=-\log r$, i.e. $\exp -m = r$ and $m=n\delta$.
Then  by Lemma \ref{uniformlyperfect} applied to $a=1/2$ for each $x\in J(f)$ there is a set $X(x)$ of at least
$\exp \kappa m/2 = \exp \kappa n\delta/2$
of $r$-separated points in $J(f)\cap B(x, \exp -m/2)$ in particular in $B(x,\exp -n\delta/2)$.
Since $n\#(\Crit(f)\cap J(f))\ll \exp \kappa n\delta$ for $n$ large enough, then for each $x=f^j(c), c\in\Crit(f)\cap J(f), j=1,...,n$, there is a point $\hat x$ in $B(x,\exp -n\delta/2)\setminus \sB$. Now we repeat the proof as in the connected $J(f)$ case, with $\hat x$ playing the role of $x^k$.

\smallskip

{\bf II.} Now we prove the opposite inequality, namely
$$
P_{\spanning}(t) \ge P_{\tree}(t).
$$
In fact we shall prove
$$
P_{\spanning}(t) \ge P_{\rm hyp}(t)
$$
which is enough due to Theorem (\ref{tree-hyp-complex}).


By \cite[Proposition 2.1]{PrzRivSmi:04} for every $\xi>0$ and $t>0$ there exists an $f$-invariant isolated hyperbolic set $X\subset J(f)$ such that
$P(f|_X,\phi_t|_X) \ge P_{\rm {hyp}}(t)-\xi$. Then for every $\e>0$ small enough for every $n\ge 0$ large enough there exists
 an $(n,2\e)$-separated set $X_n\subset X$ such that
 $\sum_{y\in X_n} |(f^n)'(y)|^{-t} \ge \exp P_{\rm{sep}}(f|_X,\phi_t|_X)-\xi$.

 Therefore for every $(n,\e)$-spanning set $Y_n\subset J(f)$ 
for every $y\in X_n$, there exists $y'\in Y_n$ which is $(n,\e)$-close to $y$.

Hence by triangle inequality the selection $y\mapsto y'$ is injective.
By the hyperbolicity of $X$, if $\e$ is small enough, there is a constant $C$ such that for every $n$ and $y\in X_n$ it holds
$|(f^n)'(y')|/|(f^n)'(y)|\le C$. This, after passing to limits, proves 
$P_{\spanning}(t) \ge  P_{\rm{sep}}(f|_X,\phi_t|_X)-\xi$. Hence letting $\xi\to 0$ and choosing appropriate $X$ we obtain $P_{\spanning}(t) \ge P_{\rm{hyp}}(t)$.

\smallskip

We considered here all $(n,\e)$-spanning sets, so it is natural to call this Part II of the proof the ALL part.
Notice that this proof corresponds to the left hand side inequality in (\ref{equal-pressures}).


\end{proof}

To end this section let us provide the lemma we have already referred to

\begin{lemma}\label{unif}
For every $t>0$ and $\xi>0$ there exists $\delta_0>0$ and $n_0\ge 0$ such that for every
$0<\delta\le\delta_0$ and $n\ge n_0$  and every
$z_1,z_2\in J(f)$ such that $z=z_i, i=1,2$ satisfies \eqref{safe1}, it holds for $Q$ defined in \eqref{Ptree}
$$
\exp -n\xi \le \frac{Q_n(z_1,t)}{Q_n(z_2,t)}\le \exp n\xi.
$$
\end{lemma}
To prove this lemma we use the fact being a part of \cite[Lemma 3.1]{Prz:99} (see also \cite{HalHay:00} and \cite[Geometric Lemma]{PrzRivSmi:03})

\begin{lemma}\label{quasihyperbolic} There exists $C>0$ such that for every set $W$ of $m>0$ points in $\widehat{\C}$ and $0<r<1/2$ such that $m\ge \log 1/r$, for every $z_1,z_2\in \widehat{\C}\setminus W$ there exists a sequence of discs in the Riemann sphere metric $B_1=B(q_1,r_1),...,B_k=B(q_k,r_k)$ such that for every $j=1,...,k$ each $2B_j:=B(q_j,2r_j)$ is disjoint from $W$, $z_1\in B_1, z_k\in B_k$,
$B_j\cap B_{j+1}\not=\emptyset$ for all $j=1,...,k-1$ and
\begin{equation}
k\le C\sqrt{m} \sqrt{\log 1/r}.
\end{equation}
In other words the quasi-hyperbolic distance between $z_1$ and $z_2$ in $\widehat{\C}\setminus W$ is bounded by $\Const\sqrt{m} \sqrt{\log 1/r}.$
\end{lemma}

\begin{proof}[Proof of Lemma \ref{unif}]
Given $n$ set $W=\bigcup_{j=1,...,n}f^j(\Crit(f))$
and 

\noindent $m=n\#\Crit(f)$. Using Lemma \ref{quasihyperbolic} and Koebe distortion lemma we
obtain for a sequence $s_j\in B_j\cap B_{j+1}$  for $j=1,...,k-1$ and $s_0=z_1, s_{k+1}=z_2$, and for a distortion constant $C_{\dist}>0$
$$
\frac{Q_n(s_j,t)}{Q_n(s_{j+1},t)}\le C_{\dist}^t,
$$
hence
$$
\frac{Q_n(z_1,t)}{Q_n(z_2,t)}\le C_{\dist}^{t(k+1)}.
$$
Hence, for $r=\exp -n\delta$, due to
$$
k\le C\sqrt{n\#\Crit(f)} \sqrt{n\delta} = C'n\sqrt{\delta}
$$
for $C'=C\sqrt{\#\Crit(f)}$, we obtain
$$
\frac{Q_n(z_1,t)}{Q_n(z_2,t)}\le C_{\dist}^{tC'n\sqrt{\delta}+1}\le \exp n\xi
$$
for $\delta$ small enough and $n$ satisfying \eqref{safe1} and large enough
to satisfy the latter inequality.

\end{proof}


\section{Weak backward Lyapunov stability and further corollaries in the complex case}

For every $x\in \widehat{\C}$ and a rational mapping $f:\widehat{\C}\to\widehat{\C}$ define the lower Lyapunov exponent by
$$
\underline\chi(x):=\liminf_{n\to\infty} \frac1n\log |(f^n)'(x)|.
$$

Let us start with the following 

\begin{proposition}\label{Prop}
For every rational mapping $f:\widehat{\C}\to\widehat{\C}$ of degree at least 2, if for  every critical point $c\in J(f)$ the lower Lyapunov exponent $\underline\chi(f(c))$ is non-negative, then  weak backward Lyapunov stability {\bf wbls} holds.
\end{proposition}

\begin{proof}
Take arbitrary $\e,\delta>0$ and $x\in J$, and consider large $n$. 
Consider $B:=B(x,\exp -n\delta)$ and an arbitrary $y\in f^{-n}(x)$. 
For every $0<j\le n$ consider $U_j=B(x,a_j \exp -n\delta)$, where $a_j=\prod_{s=1}^j (1-s^{-2})$. Let $V_j$ be the pull-back of $U_j$ for $f^j$ containing $f^{n-j}(y)$. Let $j=j_1 > 0$ be the least non-negative integer for which $V_{j+1}$ contains a critical point $c$. Then $c\in J(f)$ if $n$ is large enough.  Indeed, the only other possibility would be a critical point $c\notin J(f)$ attracted to a parabolic periodic orbit.  Then however the convergence of $f^n(c)$ to this orbit, and moreover to $J(f)$ would be subexponential, so $f^{s}(c)\notin B(x, \exp -n\delta )$ for $s=1,2,...,n$ if $n$ is large enough.

(In fact we can omit this part of the proof, since the proof below does not use $c\in J(f)$. It uses only $\underline\chi(f(c))\ge 0$, true if $f^n(c)$ converges to a parabolic periodic orbit.)

\smallskip

Then, for diameters and derivatives in the spherical metric,
\begin{equation}\label{shrinkingneighb}
\frac{\diam f(V_{j+1})}{\diam U_{j+1}}\le C_1 (j+1)^{-8} |(f^j)'(f(c))|^{-1}\le C_2\exp j\xi.
\end{equation}
The term $C_1(j+1)^{-8}=C_1((j+1)^{-2})^4$ results from Koebe's distortion bounds, see e.g. 
\cite[Lemma 6.2.3]{PrzUrb:10} for the spherical setting. The `isolating annulus' is
$$
a_{j+1} \exp -n\delta<|z-x|< a_j \exp -n\delta.
$$
This method of controlling distortion was introduced in \cite[Definition 2.3]{Prz:98} and developed and called in \cite{GraSmi:98}  {\it shrinking neighbourhoods}. $j=j_1$ is called the {\it first essential critical time}.

In \eqref{shrinkingneighb} $\xi>0$ is arbitrary and $C_2$ is an appropriate constant.
The latter inequality follows from the assumption that

\noindent $\underline\chi(f(c):=\liminf_{n\to\infty} \frac1n\log |(f^n)'(f(c))|\ge 0$. In the sequel we shall assume that $\xi\ll \delta$.

\smallskip

Consider now $B_0=B(x,\kappa \exp -n\delta)$, for $0<\kappa\ll 1$ small enough that $B_0$ is deeply in $B(x,\prod_{s=1}^\infty(1-s^{-2}) \exp -n\delta)$ so that for the associated pull-backs $W_t$ of $B_0$, for $t=1,2, ... ,j$ we have $\diam W_t\le \e$. This is possible due to bounded distortion before the capture of $c$. 

\smallskip

Now notice that for $j=j_1\ge N$ for a constant $N=N(\delta,\xi)$ by  \eqref{shrinkingneighb} 
\begin{equation}\label{diam1}
\diam W_{j_1}\le \exp (-n\delta+j_1 2\xi) \le \kappa \exp (-n\delta+j_1 \delta/2).
\end{equation}

Hence,  denoting by $\tau=\tau(c)$ the multiplicity of $f$ at the critical point $c$,
\begin{equation}\label{diam2}
\diam W_{j_1+1}\le \kappa\exp \frac1{\tau(c)} (-n\delta+j_1\delta).
 \end{equation}

 If $j_1<N$ then we
 obtain \eqref{diam2} automatically if we replace $B_0$ by a disc centered at $x$ of diameter $a \exp (-n\delta)$ with $a$ small enough.

\smallskip

Denote $n_1=n-j_1-1$.
Apply the shrinking neighbourhood procedure starting from
$B(f^{n_1}(y),\exp -n_1\delta/\tau)$.
Let $0<j_2\le n_1$ be the first essential critical time, if it exists.
Denote the captured critical point by $\hat c$ (it can be different from the former $c$)

Denote $B_1=B(f^{n_1}(y),\kappa \exp -n_1\delta/\tau)$.
Notice that $B_1\supset W_{j_1+1}$.
Denote the consecutive  pull-backs of $B_1$ by $W_{1,j}$.  Repeating \eqref{shrinkingneighb} we obtain,
analogously to \eqref{diam1}, using an analogon of \eqref{one-step}, 
\begin{multline}
\diam W_{1,j_2}\le
C_1 (j_2+1)^{-8} |(f^{j_2})'(f(\hat c))|^{-1} \diam B_1 \le \\
\Const C_1 (j_2+1)^{-8} |(f^{j_2})'(f(\hat c))|^{-1} |f'(f^{j_2}(f(\hat c)))|^{-1} \diam f(B_1) \le \\
C_1 (j_2+1)^{-8} \exp (j_2+1)2\xi \diam f(B_1) \le \\
\kappa \exp (-n\delta + j_1\delta + (j_2+1)\delta/2),
\end{multline}


Here we have $j_2> N(\delta,\xi)$ automatically, provided $c=\hat c$ and $n_1$  large enough, since  otherwise $c$ is periodic attracting hence not in $J(f)$.

Denote $n_2=n_1-j_2-1$ and
continue, choosing $j_3, j_4, ...$, until an essential critical time $j_k$ does not exist; then the last
pull-back is just the pull-back of $B_{k-1}\ni f^{j_k}(y)$ for $f^{j_k}$, containing $y$, \ $j_k\ge 0$.
By the `telescoping' construction and isolating annuli of moduli $\log \Const\kappa$, all the pull-backs $W_s$ of $B_0, s=1,...,n$ have diameters not exceeding $\e$.

If there is more than one critical point in $J(f)$ then the proof should be modified in a standard way.
It relies on the observation that for $n$ large enough the pull-backs under consideration have small diameters so $j_s$ is small only if
$f^{j_s+1}(\hat c)=c$ which can happen consecutively only $\#(\Crit(f)\cap J(f))$ number of times, otherwise a critical point in $J(f)$ is periodic.

\end{proof}

From Proposition \ref{Prop} and Theorem \ref{spanning-complex} it follows

\begin{theorem}\label{Mcomplex2}
For every rational mapping $f:\widehat{\C}\to\widehat{\C}$ of degree at least 2 such that for  every critical point $c\in J(f)$ the lower Lyapunov exponent $\underline\chi(f(c))$ is non-negative, and for every $t>0$,  the equality
$
P_{\spanning}(t) = P_{\tree}(t)
$
holds.
\end{theorem}

Now let us invoke the following part of \cite[Theorem 5.1]{LPS:14}

\begin{theorem}
For every rational mapping $f:\widehat{\C}\to\widehat{\C}$ of degree at least 2, such that there is exactly  one critical point
$c$ whose forward orbit has an accumulation point in $J(f)$ (i.e. $c\in J(f)$ or the forward trajectory of $c$ being attracted to a parabolic periodic orbit), we have $\underline\chi(f(c))\ge 0$.
\end{theorem}

This and Theorem \ref{Mcomplex2} yield

\begin{corollary}
Let $f:\widehat{\C}\to\widehat{\C}$ be a rational mapping of the Riemann sphere of degree at least 2, such that there is at most one critical point whose forward trajectory has an accumulation point in $J(f)$, then
$P_{\spanning}(t) = P_{\tree}(t).$
\end{corollary}

\

Without the assumption of weak backward stability, i.e. in the full generality, we can prove
only the following in place of Theorem \ref{spanning-complex}

\begin{theorem}\label{spanning-finite}
For every rational mapping $f:\widehat{\C}\to\widehat{\C}$ of degree at least 2 and for every  $t>0$
$$
P_{\spanning}(t)<\infty.
$$
\end{theorem}

\begin{proof}
We proceed as in the proof of Theorem \ref{spanning-complex} Part I, with small modifications.
Notice that there exists $\Delta>0$ such that for an arbitrary $\e>0$  we have  for $n$ large enough for every $x\in J(f)$ and every pull-back $V$ of $B(x,\exp(-n\Delta/2))$ for $f^j, j=0,...,n$,\ $\diam V<\e$.
This fact follows  immediately from \cite{DPU:96}[Lemma 3.4].

Denote  $\sB:=\bigcup_{c\in\Crit(f)\cap J(f)}\bigcup_{j=1,...n} B(c,j)$,
where $B(c,j):=B(f^j(c), r)$ where $r:=\exp (-n\Delta))$. Then we find $X \subset J(f)\setminus \sB$ which is $r/2$-spanning and $\#X\le \Const \exp 2n\Delta$. Then we find an $(\e,n)$-spanning set $Y$ as in the proof of Theorem \ref{spanning-complex} Part I. Finally, in place of the inequality (\ref{key-spanning}), we just estimate $|(f^n)'(y)|^{-t}$ for $y\in Y$. For this aim we shall use the following, see \cite{DPU:96}[Lemma 2.3]
\phantom\qedhere\end{proof}
\begin{theorem}[Denker, Przytycki, Urba\'nski]\label{DPU} 
For every rational mapping $f:\widehat{\C}\to\widehat{\C}$ of degree at least 2
there exists $C_f>0$, which depends only on $f$
such that for every $z\in J(f)$
\begin{equation}\label{psibound}
\sum_{k=0}^{n-1}{\bf '} \ \varphi(k)\le C_f n, \ \ \ n=1,2,...
\end{equation}
where $\varphi(k)=-\log\rho(f^k(z),\Crit(f))$ and $\sum {\bf '}$ denotes the summation over all but at most $M=\#\ Crit$ indices.
\end{theorem}

\

\begin{proof}[Continuation of Proof of Theorem \ref{spanning-finite}]
We can now write, using (\ref{psibound}),
$$
|(f^n)'(y)|^{-t}\le \exp (tC n) \exp (tC\Delta  M n) L^{tMn}.
$$
for a constant $C\ge 0$ depending on $C_f$ and the multiplicities of the critical points. The factor 
$\exp(tC\Delta Mn) L^{tMn}$ takes care of (at most) $M$ integers $k$ omitted in $\sum{\bf '}$. For these $k$ we use $\rho(f^k(y),\Crit(f))\ge L^{-(n-k)}\exp(-n\Delta)$, where $L=\sup|f'|$, true since otherwise
$\ddist(f^n(y),f^{n-k}(\Crit(f)))<$ $\exp -n\Delta$ which contradicts the definition of $r$ in $\sB$ above.

Hence, collecting our estimates,
$$
P_{\spanning}(t)\le h_{\rm{top}}(f)+tCM\Delta + tM\log L.
$$
 \end{proof}

\


\section{Geometric pressure via spanning sets. The real case}

We start from  a notion refining the definition of safe, see Definition\ref{safe}

\begin{definition}\label{fold safe} For $(f,K)\in {\sA}$ a point $z\in K$ is called {\it safe from outer folds} if
for every $\eta>0$ and all $n$ large enough,
for every pull-back $W_n$ of $W=B(z, \exp (-\eta n))$ for $f^n$,
intersecting $K$, containing a turning critical point for $f^n$,
 there is a point $z_n \in \partial W_n$
 such  that  $f^j(z_n)\in \hat I$ for all $j=0,1,...,n$.
\end{definition}

\begin{theorem}\label{spanning-real}
For every $(f,K)\in {\sA}^{\BD}_+$ (or ${\sA}^3_+$) without parabolic periodic orbits, weakly isolated, for every  $t>0$ and
every safe $z\in K$ it holds
$P_{\spanning}(t) \ge P_{\tree}(z, t)$.

If every periodic $z\in\partial \hat I$ is safe from outer folds, then the equality of the pressures holds. In particular it holds, provided $K=\hat I=I$, namely it is a single interval

\end{theorem}

\begin{proof}

\smallskip

I. The CONSTRUCTION inequality.

We mostly repeat parts of the proof of Theorem \ref{Independence-safe}

Fix an arbitrary safe $z\in K$ and $\delta$ is chosen to $\e$ as in the
Definition of backward Lyapunov stability.

For an arbitrary $0<\delta'\le \delta$ let
$N=N(\delta')$ be such that

\noindent $A=A(z,\delta'):=\bigcup_{j=0,...,N}f^{-j}(z)\cap K$ is $\delta'$ dense in $K$.


We shall prove that the set $f^{-n}(A)\cap K$ itself happens to be
an $(\e,n)$-spanning set (roughly, for $w$ with $f^n(w)$ as in  case (i) below) under an assumption as in  Theorem \ref{spanning-real}.
Then immediately $P_{\spanning}(t) \le P_{\tree}(z, t)$.

\smallskip

Indeed, if $w'=f^{n}(w)\in  W=[z_0,z_0']$ with its endpoints belonging to $A$ whose distance is at most $\delta$ then for its pull-back $W_{n}=[z_{n},z'_{n}]$ containing $w$ we have for all $j=0,...,n$,
$|f^j(z_{n})-f^j(w)|<\e$ (and the same for $z'_{n}$). The proof that $z_{n}$ or $z'_{n}$ is in $K$ is the same as in Proof of Theorem \ref{Independence-safe} and uses the weak isolation assumption.

\smallskip

A trouble is with $w$ such that $w'=f^{n}(w)$ is not in any $W$ as above.
Then, as in   Proof of Theorem \ref{Independence-safe} there is a large gap (a component in $\R\setminus K$) of length at least $\delta/4$ within the distance at most $\delta'$ of $w'$.

Then we have two cases.

(i) For some $m$ bounded by a constant depending only on $(f,K)$ and $\delta$ the point $f^m(w')$
belongs to some $W$ with endpoints $z_0,z'_0 \in A(z,\delta')$ for $\delta'$ satisfying (\ref{delta'}). Then $\rho_{n}(w,z_{n+m}) \le \rho_{n+m}(w,z_{n+m})< \e$ for an appropriate $z_{n+m}$ in 
 the boundary of the pull-back of $W$ for $f^{n+m}$ containing $w$.

\smallskip

(ii) For some $n+m$ the point $w''=f^{n+m}(w)$ is close to a
periodic point $p$ in the boundary of a large gap $G$.

Notice that in fact $p\in\partial\hat I$, see \cite{PrzRiv:13}[Lemma 2.9, Case 2]. Indeed, if $p$ and all other points of its periodic orbit belong to the interior of $\hat I$, then also
$\bigcup f^j(G)\subset\interior\hat I$. Otherwise, if $j_0$ is the least integer such that $f^{j_0}(G)$ intersects $\partial \hat I$ at a point $y$, then $y_{j_0}=(f^{j_0}|_G)^{-1}(y)\in G$ and it belongs to $K$ since $y\in K$ and all $f^j(y_{j_0})$ belong to $\hat I$ hence to $K$ by the maximality of $K$.
 This contradicts $G\cap K=\emptyset$. If all $f^j(G)$ are in $\hat I$, then by the maximality $G\subset K$, a contradiction.

\smallskip

 Then, as at the end of Proof of Theorem \ref{Independence-safe}, consider $\hat z\in f^{-[\kappa n]}(z)$
 belonging to  $B=B(p, \exp -\eta n) \cap K$, for $\eta<\kappa\chi(p)$ where $\chi(p)$ is Lyapunov exponent at $p$. In particular $|p-\hat z|<|p-\exp \eta n|$. 
 Denote $r=|p-\hat z|$ and $B':=B(p,r)\subset B$.
 
 If $w''\notin B'$, then for some $k \le \kappa\chi n $ the point $v=f^k(w'')$ is far from the periodic orbit of $p$ but $f^k$ is still invertible on $B(p,|p-w''|)$. In particular there exist $z_0, z'_0\in A$ such that
 $|z_0-z'_0|<\delta$ and $v\in [z_0,z'_0]$. Hence $w''\in [z_k,z'_k]$, the pull-back. 
 Hence, as before, $w\in [z_{n+m+k},z'_{n+m+k}]$ where one of the ends say $z_{n+m+k}$ is in $K$ and 
 $\rho_n(w,z_{n+m+k})<\e$. 
 
 If $w''\in B'$, then 
 by the assumption that $p$ is safe from the outer fold for the constant $\eta$ for $n$ large enough, for $[z(w),z'(w)]$ being the pull-back of $B$ for $f^{n+m}$ containing $w$,  all $f^j(z(w)), j=0,...,n+m$ belong to $\hat I$ (or the same for $z'(w)$). In particular $u:=f^{n+m}(z(w))$ is the point of $\partial B$ in $\hat I$.
  
  By our definitions, $\hat z$ is between $w''$ and $u$. Since $w\in K$, $f^j(w)\in \hat I$ for all $j\ge 0$. Hence there exists $\hat z_{n+m}\in [w,z(w)]\cap f^{-n-m}(\hat z)$.
   such that $f^j(\hat z_{n+m})\in \hat I$ as belonging to $f^j([w,z(w)]$ being intervals shorter than $\e$ with ends in $\hat I$.  These ends may be of the form $f^j(w), f^j(z(w))$ or $f^i(c)$ for a turning critical point $c\in K$ hence in $K\subset \hat I$.
  
  Hence $\hat z_{n+m}\in K$ and 
 $\rho_n(w,\hat z_{n+m})\le\e$.


 \smallskip

 So, given $\e>0$ and safe $z\in K$, for all $\kappa>0$, for all $n$ large enough,
 the set
 $$
 {\rm SP}(z,n):= \bigcup_{0\le j \le \kappa n + \Const(\e)} 
 f^{-n-j}(\{z\})
 $$ 
 is
 $(n,\e)$-spanning. 
 $\Const(\e)$ depends on $N$ and $m$ above which depend on $\e$.
 
 Next use $|(f^j)'|\le L^j$. 
 We have, denoting $\hat n=\kappa n + \Const(\e)$, for $\xi>0$,  
 
 $$
 \sum_{x\in {\rm SP}(z,n)} |(f^n)'(x)|^{-t} \le \sum_{0\le j \le \hat n} Q_{n+j}(z,t)L^j
 $$
 $$
 \le L^{\hat n}\sum_{0\le j \le \hat n} \exp \bigl((n+j)(P_{\tree}(z)+\xi)\bigr).
 $$
 
 For $n\to\infty$ and $\kappa\to 0$ this holds for $\xi$ arbitrarily small and finally 
 $P_{\spanning}(t) \le P_{\tree}(z, t)$.


\smallskip

Notice that unlike in Proof of Theorem \ref{Independence-safe} we have not needed here to compare the derivatives $|(f^n)'(w)|$ and the shadowing $|(f^n)'(z_n)|$. In particular we consider all $w$, rather than having $f^n(w)$ safe.

\smallskip

Notice finally that if $K=\hat I=I$ is a single interval, then every $z\in \hat I$ is safe from outer folds. Otherwise both ends $z_n,z_n'$ of $W_n$ are outside $\hat I$, since if, say, $z_n\in \hat I$ then
 all $f^j(z_n)\in \hat I$ by the forward invariance of $K=\hat I$ here. So $z_n$ and $z'_n$ are on the  different sides of $I$.
This is not possible since $W_n$ is short by backward Lyapunov stability of $f$.

\smallskip

II. The ALL inequality.  The proof is the same as in the complex case, via $P_{\spanning}(t)\ge P_{\rm{hyp}}(t)$.

\end{proof}

\

\begin{example} We show that the assumption on the safety from outer folds is really needed in Theorem \ref{spanning-real} above.

\smallskip

$\bullet$ \ Consider quadratic polynomials $f_a(x)=ax(1-x)$ for $0<a<4$ large enough that the entropy of $f_a$ is positive. For each $a$ let $p_a$ denote the unique fixed point in the open interval $(0,1)$. It is repelling; let us make a small perturbation of $f_a$ close to $p_a$ so that $p_a$ becomes attracting and a repelling orbit $Q_a$ of period 2, being the boundary of $B_0(p)\subset (1/2,1)$ which is the immediate basin of attraction to $p_a$, is created.

One can do it in such a way that Schwarzian derivative $Sg$ of 
the new map $g=g_a$ is negative except in $B_0(p)$.  Write $Q_a=\{q_a,q'_a\}$ with $q_a<q'_a$. Omit the subscript $a$. Define 
$$
g (x)=\begin{cases} f(x)-b \times (x-q)^3, & {\rm{if}} \ \   q<x<p-\e          ;\\ 
          f(x)+ b \times (q'-x)^3,   & {\rm{if}}  \ \  p+\e<x<q'         ;\\
p   & {\rm{if}}  \ \   p-\e \le x \le p+\e                                   ;\\
f(x), & {\rm{otherwise}} .\end{cases}
$$
One can choose $\e>0$ arbitrarily small and $b>0$ so that the above function is continuous. Then $b$ is also small hence by $Sf<0$ Schwarzian of $g$ stays negative except in $[p-\e,p+\e]$.

\smallskip

$\bullet$ \ Let
$\underbar I_a=(I_n)_{n=1,2,...,N}$ denote the kneading sequence for $g_a$, that is the sequence of letters $L,R,C$ depending whether $c_n=g_a^n(1/2)$ lies to the left of the critical point $1/2$, to the right of $1/2$, or at $1/2$. We put $N$ the least integer $n$ for which $I_n=1/2$. If no such integer exist we put $N=\infty$. See \cite{CE:80} for these definitions.

Let
\begin{equation}\label{itinerary}
\underbar I = RLR^{n_1}LR^{n_2}LR^{n_3}L...,
\end{equation}
where $N=\infty$, all $n_j$ are finite positive, even, their sequence is increasing and $n_j\to\infty$ exponentially fast as $j\to\infty$.

$ \underbar I$ is a maximal sequence for every sequence $(n_j)$ satisfying above conditions, hence there exists $a$ such that $g=g_a$ has this kneading sequence, see \cite{CE:80}[Theorem III.1.1].

For  $\underbar I$ as above for $c_n$ left of $1/2$ we have $c_{n+1}$ right of $1/2$ and close to $q_a$, left of it (remember $Q_a=\{q_a,q'_a\}$ with $q_a<q'_a$). Next the trajectory $c_{n+2}, c_{n+3}, ...$ follows $Q_a$ outside of $[q_a,q'_a]$, in the interval $(1/2,1)$ until $c_{n+n_k+1}$ occurs to the left of $1/2$ moreover to the left and close to the point symmetric to $q_a$ with respect to $1/2$.

\smallskip

$\bullet$ \ Now consider $\hat I=[c_2,q_a] \cup [q'_a,c_1]$, $g$ restricted to a neighbourhood $\bf U$ of $\hat I$ and $K$ the maximal forward invariant subset of $\hat I$. Clearly $1/2\in K$ since otherwise $g^n(1/2)\to p$
so $\underbar I$
would consist solely of $R$'s for $n$ large enough. $K=\hat I\setminus B(p_a)$, where $B(p_a)$ is the basin of attraction by $g$ to $p_a$. Due to $Sg<0$ on  a neighbourhood ${\bf U}$ of $K$ we obtain 
 $(g,K,\hat I, {\bf U})\in \sA^{\rm{BD}}$, provided we prove 

\smallskip

$\bullet$ \ {\bf Claim}: $g$ is topologically transitive on $K$. 

\smallskip

  Let $a'$ be so that the kneading sequence for $f=f_{a'}$ is the same as for $g_a$, that is 
  $\underbar I$.
Due to the lack of attracting and parabolic periodic orbits for $f_{a'}$ (otherwise $\underbar I$ would be eventually periodic), there is
a monotone increasing continuous semiconjugacy $h:[c_2,c_1]\to [c_2,c_1]$ such that $f\circ h= h\circ g$. $h$ is defined first in a standard way on $\sT (g):=\bigcup_{n\ge 0} g^{-n}(1/2)$  to the corresponding
$\sT (f):=\bigcup_{n\ge 0} f^{-n}(1/2)$, 
 increasing since the orders in the interval $[c_2,c_1]$ of points in these sets are (combinatorially) the same, due to the same kneading sequences.

 The mapping $h$ can be continuously extended to the closures, and notice that $\cl\sT (f)=[c_2,c_1]$ due to the absence of wandering intervals for $f$.
 
 This $h$ collapses  $B_0(p)$ and its $g^n$-preimages to points, provided we extend $h$ to these gaps by constant functions. In other words $h$ identifies the pairs of points being ends of  gaps $B(p)$ being components in the basin of $p$. There are no other gaps in $[c_2,c_1]\setminus \cl\sT (g)$  since there are no wandering intervals (see \cite{dMvS}) and no attracting or parabolic periodic orbits other than $p$. This in turn holds  since the Schwarzian
$Sg$ is negative outside $B(p)$ so the basin of such an orbit would contain a critical point that is $1/2$ which is not possible since $\underbar I$ is not eventually periodic. Therefore $h$ is injective on $K$ except
the abovementioned pairs of points.  

Notice that our $\underbar I$ is not a *-product, see \cite{CE:80}[Section II.2] for the definition.
Hence there is no interval $T\subset I_f=[f^2(1/2),f(1/2)]$ such that $f^k(T)\subset T$ for some $k> 1$ containing $1/2$ with $f^k$ unimodal on it (i.e. with one turning point), i.e. there is no renormalization interval. (In other words $f$ is not renormalizable). This follows from \cite{CE:80}[Corollary II.7.14].

Consider now any interval $T\subset I_f$ and $V=\bigcup_{j\ge 0}f^j(T)$. By definition $V$ is forward invariant. Let $W$ be a connected component of $V$. Then there are integers
$0\le k_1<k_2$ such that $f^{k_1}(W)\cap f^{k_2}(W) \not=\emptyset$ since $W$ is non-wandering, see \cite{dMvS}[Chapter IV, Theorem A] for the non-existence of wandering intervals. Hence, for $k=k_2-k_1$,
and $W'=f^{k_1}(W)$,  $f^k(W')\subset W'$. We consider $k$ the smallest such integer. We can assume that $1/2\in W'$ (or some $f^j(W')$), since otherwise $W$ would be attracted to a periodic orbit and  we have assumed such orbits do not exist. No $f^\ell(W'), 0<\ell<k$ contains $1/2$ by its disjointness from $W'$. 
So $f^k$ is unimodal on $W'$. So $k = 1$, since otherwise $f$ would be renormalizable. So $f(1/2)$ and $f^2(1/2)$, the end points of $I_f$, belong to $W'$.
Hence $V=I_f$, hence $f$ is topologically transitive on $I_f$.

This due to our semiconjugacy and the fact that $K$ has no isolated points, implies the topological transitivity of $g$ on $K$. The Claim is proved.

\smallskip

The property we proved in particular, $(\forall\ {\rm{open}} \ W) \ (\exists k),\ f^k(W)={\rm{domain}}(f)$ is called topological exactness or leo -- "locally eventually onto". This is stronger than topological transitivity. See \cite{PrzRiv:13}[Lemma A7] for a discussion of a general case. Since the topological entropy $h_{\rm{top}}(g|K)>0$ we can write $(g,K,\hat I, {\bf U})\in \sA^{\rm{BD}}_+$

\smallskip

$\bullet$ \ Notice that $K$ is weakly isolated for $g$ on $\bf U$, see Definition \ref{weak isolation}. This is so because if a periodic trajectory $P$ in $\bf U$ has a point $z\notin K$ then $z$ belongs to the basin of attraction to $p$, i.e. $g^n(z)\to p$. In other words the trajectory $g|_{\bf U}^n(z)$ leaves $\bf U$.
Hence $P\subset K$. 
Note that above argument proves the  weak isolation property in general situations, namely if $K$ is Julia set in the sense of
\cite{dMvS}[Chapter IV, Lemma] i.e. the domain being an interval with the basins of attracting or parabolic periodic orbits removed (provided there is a finite number of them).

\smallskip

$\bullet$ \ Notice that $q_a$ is not safe from outer folds, see Definition \ref{fold safe}. Indeed.
Denote $2+\sum_{j=1,...k} n_j +k +1$ by $m_k$. The summands $n_j$ correspond to the blocks of $R$'s,
the first summand 2 corresponds to the starting $RL$ and the final 1 to the first  $R$ in the $k+1$'th block of $R$'s. We obtain $|c_{m_k}-q_a|\le \Const(a) \exp -n_{k+1}\chi(q_a)$, where $\chi(q_a)= \frac12 \log |(g^2)'(q_a)|$. Consider the pull-back $W_{m_k}$ of $W=B(q_a, \exp (-\eta m_k))$ for $g^{m_k}$ containing $\frac12$.

\noindent \underline{The critical point $1/2$ is not recurrent} since $c_{m_k-1}$ corresponding to $L$ approach to the point symmetric to $q_a$ since $n_j$ grow, so they are in the distance from $c_2$ bounded away from 0. Hence the only points we need to care about, $c_{m_k-3}$ are in the distance from $1/2$ also bounded away from 0. 

Hence for all $n=1,2,...m_k-1$ the map $g$ is injective on $g^n(W_{m_k})$ but $g$ has a turning critical point $1/2$ in $W_{m_k}$. Using the assumption that all $n_j$ are even we conclude that each $g^{m_k}$ has a minimum at $1/2$, hence if $n_{k+1}\chi(q_a) \gg m_k\eta$ the boundary points of $W_{m_k}$ are mapped by $g^{m_k}$ into the gap (basin $B_0(p)$). In other words $\partial W_{m_k}\subset B_{m_k}\cup B_{m_k}'$ the latter being the union of a symmetric pair of pull-backs of $B_0(p)$ for $g^{m_k}$ on both sides of $1/2$.  

\smallskip

$\bullet$ \ Imposing sufficient growth of $n_j$, e.g.
\begin{equation}\label{fast}
n_{j+1}/n_j\to\infty,
\end{equation}
 we get a counterexample to $P_{\spanning}(t) \le P_{\tree}(z, t)$.
Indeed,
consider an arbitrary $x\in K$ which $(n,\e)$-close to $1/2$ for $n=m_k$.

Due to the non-recurrence of $1/2$, see above,  $g$ is expanding on the limit set $\omega(1/2)$, see Definition \ref{hyperbolic} and e.g. Ma\~{n}\'e's theorem: \cite{dMvS}[Section III.5 Corollary 1]. 
Denote the expanding constant by $\lambda$, compare Definition \ref{hyperbolic}. 

Hence for $\e$ small enough and an integer $N$ such that all $g^j(1/2), j\ge N$
are close to $\omega(1/2)$, if $|g^j(1/2)-g^j(x)|\le \e$ for all $j:N\le j\le n$, then for all $N\le j \le n$, $g^j(x)\in g^j(W_n)$, where $W_n$ is the pull-back of $W$ as above, but for
$W=B(q_a,\e)$ (unlike above). Then this holds also automatically also for $0\le j < N$, maybe on the cost of taking a smaller $\e$.

Suppose $n_{j+1}\gg m_j$. Then $\diam g^n(W_n)\cap \hat I\le \exp -C n$ for $C$ large.  Hence for all $0<j\le n$, $|g^j(x)-c_j|\le \exp -Cn$ and $|x-1/2|\le \exp -Cn/2$. Hence $|(f^n)'(x)|\le 
 \lambda^{2n}\exp -Cn/2$.
Hence, for every $(n,\e)$-spanning set $Y\subset K$, $\sum_{y\in Y} |(g^n)'(y)|^{-t}\ge \lambda^{-t2n}\exp Cn/2$. The assumption \eqref{fast} allows to have $C$ arbitrarily large.

We conclude that  $P_{\spanning}(t) =\infty$.

\end{example}

\

\

\begin{figure}[ht]
\def\svgwidth{12cm}
\centering
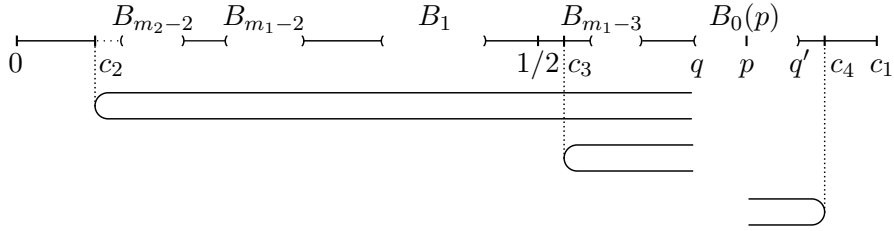
\caption{Pull-backs of $B_0(p)$.}
\end{figure}


\

\begin{remark}

\smallskip

1. In the example above $K$ is not uniformly perfect (considered in the plane), unlike in the complex case where the uniformly perfect property of Julia set allowed us to prove Theorem 4.

\smallskip

2. In this example the so-called Bowen's periodic specification property does not hold. This property is defined
for any continuous map $f:X\to X$ of a compact $X$ as follows: For every $\e>0$ there exists an integer $N$ such that for every $x\in X$ and every integer $n\ge 0$ there exists $y\in X$ of period
$k:n\le k\le n+N$ such that for every $0\le j\le n$, $\ddist (f^j(x),f^j(y))\le \e$.

Even a weaker periodic specification does not hold, where $N=N(\e)$ is replaced by $N(n,\e)$ for $\e$ small enough (see the survey \cite{KLO:15}). Namely for every function $N(n,\e)$ there exists $a$ such that for $g_a$ with an appropriate kneading sequence $\underline I$ the specification with $N(\e,n)$ does not hold. Consider blocks of the $g$-trajectories $1/2, c_1,c_2,... c_{m_j}$ with $n_k$ growing fast enough. Then for every $y$ being $(m_j,\e)$-close to $1/2$, $y$ is in fact $\xi$-close to $1/2$ for $\xi>0$  arbitrarily small, depending on $n_{j+1}$. Then the period of $y$ must be long since otherwise $y$ would be an attracting periodic point.

3. One can have an additional insight in the topological dynamics of $g_a$ or $f=f_{a'}$
if one uses the existence of
a semiconjugacy of $f$ to a tent map $\tau$ (of slopes $\pm h_{\rm{top}}(f)$, see \cite{MT:88}[Theorem 7.4], which must be a conjugacy since $f$ has no renormalization or wandering interval 
\cite{dMvS}[Chapter IV, Theorem A].

\end{remark}

\bibliographystyle{amsplain}

\end{document}